\documentclass[a4paper,11pt]{amsart}

\usepackage{multicol,color}
\usepackage{amsmath,latexsym,amsbsy,amssymb}
\usepackage{enumerate}
\usepackage{amsthm}
\usepackage[latin1]{inputenc}

\newtheorem{Theorem}{Theorem}[section]
\newtheorem{Definition}[Theorem]{Definition}

\newtheorem{Lemma}[Theorem]{Lemma} 
\newtheorem{Remark}[Theorem]{Remark}

\numberwithin{equation}{section}

\newcommand{\R}{{\mathbb R}}

\begin{document}

\title[Weak Sharp Solutions for Variational Inequalities]{Finite Convergence Analysis and Weak Sharp Solutions for Variational Inequalities}

\author[S. Al-Homidan, Q. H. Ansari, L. V. Nguyen]{Suliman Al-Homidan, Qamrul Hasan Ansari, Luong Van Nguyen}
\address[Suliman Al-Homidan]{Department of Mathematics \& Statistics, King Fahd University
of Petroleum \& Minerals, Dhahran, Saudi Arabia }

\email{homidan@kfupm.edu.sa}

\address[Qamrul Hasan Ansari]{Department of Mathematics, Aligarh Muslim University,
Aligarh 202 002, India, and Department of Mathematics \& Statistics, King Fahd University
of Petroleum \& Minerals, Dhahran, Saudi Arabia }

\email{qhansari@gmail.com}

\address[Luong V. Nguyen]{Institute of Mathematics, Polish Academy of Sciences, ul. \'Sniadeckich 8, 00 - 656  Warsaw, Poland \\Telephone: +48 22 5228 235 }

\email{vnguyen@impan.pl; luongdu@gmail.com}


\keywords{Variational inequalities; Weak sharp solutions;
Gap functions; Finite convergence analysis;
Exact proximal point method; Gradient projection method.}


\subjclass[2000]{49J40; 90C33; 49J53; 47J20; 90C25; 47H04}
\maketitle


\begin{abstract}
In this paper, we study the weak sharpness of the solution set of variational inequality
problem (in short, VIP) and the finite convergence property of the sequence generated by 
some algorithm for finding the solutions of VIP. 
In particular, we give some characterizations of weak sharpness of the solution set of VIP 
without considering the primal or dual gap function. 
We establish an abstract result on the finite convergence property
for a sequence generated by some iterative methods.
We then apply such abstract result to discuss the finite termination property of the sequence
generated by proximal point method, exact proximal point method and gradient projection method.
We also give an estimate on the number of iterates by which the sequence converges to a
solution of the VIP.
\end{abstract}


\section{Introduction}


Burke and Ferris \cite{BF93} introduced the concept of weak sharp solutions for an optimization problem
in terms of a gap function and gave its characterization in terms of a geometric condition.
Marcotte and Zhu \cite{MZ98} exploited that geometric condition to introduce the concept
of weak sharp solutions for variational inequalities. They also gave a characterization of
weak sharp solutions in terms of a dual gap function for variational inequalities.
It is further studied by Wu and Wu \cite{WW04}.
Recently, Liu and Wu \cite{LW15} studied weak sharp solutions for the variational inequality
in terms of its primal gap function. They also characterized the weak sharpness of the solution set
of the variational inequality problem in terms of the primal gap function.
They also presented some finite convergence results of algorithms for the VIP.
One of the most important features to study the weak sharpness of the solution set of the variational
inequality problem is that it provides the finite convergence property to the sequences generated
by the algorithms for finding the solution of variational inequalities, see, e.g.,
\cite{LW15,MZ98,MX13,MX15,WW04}.

In this paper, we give some characterizations of weak sharp solutions for the VIP without considering the
primal or dual gap function. Our characterizations give some better estimates for the distance
from any point in the underlying space to the
solution set of the VIP than the results obtained by using primal or dual gap function.
We study some abstract results on the finite termination property
for a sequence generated by some iterative methods for finding the solutions of the VIP.
As applications of the abstract results, we discuss the finite termination property of the sequence
generated by the proximal point method, exact proximal point method and gradient projection method.
We also give an estimate on the number of iterates by which the sequence converges to a
solution of the VIP.


\section{Preliminaries}


Let $H$ be a real Hilbert space whose inner product and norm are denoted by $\langle .,.\rangle$ and
$\|\cdot \|$, respectively. We denote by $\mathbb{B}$ the unit ball in $H$.
For a given set $C$ in $H$, we denote by $\mathrm{int}C$ the interior of $C$ and
by $\mathrm{cl}C$ the closure of $C$. The \textit{polar} $C^{\circ}$ of $C$ is defined by
$$C^{\circ} := \left\{ x^{*} \in H : \langle x^*,x\rangle \le 0 \text{ for all } x \in C \right\}.$$
For a given $x\in H$, the distance from $x$ to $C$ is defined by
$$\mathrm{dist}(x,C) := \inf_{y \in C} \| y-x \|,$$
and the \textit{projection} of $x$ onto $C$ is defined by
$$P_C(x) := \{ y \in C : \| y-x \| = \mathrm{dist}(x,C) \}.$$
It is well-known that $P_{C}(x)$ is a singleton set if $C$ is nonempty closed and convex.
In this case, $P_{C}$ is nonexpansive mapping, that is,
$$ \| P_{C}(x) - P_{C}(y) \| \leq \| x-y \|, \quad \mbox{for all } x,y \in C.$$
Let $X$ be a nonempty closed convex subset of $H$. The \textit{tangent cone}
to $X$ at a point $x\in X$ is defined as
$$T_X(x):= \mathrm{cl}\left( \bigcup_{\lambda >0} \frac{X-x}{\lambda}\right).$$
The \textit{normal cone} to $X$ at $x\in X$ is defined by $N_X(x) := [T_X(x)]^{\circ}$.
In other words,
$$ N_X(x) = \left\{ x^{*} \in H : \langle x^{*}, y-x \rangle \le 0 \mbox{ for all } y \in X \right\}.$$

Let $F : X \to H$ be a mapping. The \textit{variational inequality problem} (in short, VIP) is to find
$x^{*} \in X$ such that
\begin{equation}\label{VIP}
\langle F(x^*),x-x^*\rangle \ge 0, \quad \text{for all } x \in X.
\end{equation}
We denote the solution set of the VIP by $X^*$.
Throughout the paper, we assume that $X^{*}$ is nonempty.
For further details on variational inequalities and their applications, we refer to
\cite{ALM14,FP03} and the references therein.

We often consider the VIP with $F$ satisfying certain monotonicity properties.
Therefore, we recall the following definitions of different kinds of monotonicities.

\begin{Definition}
The mapping $F : X \to H$ is said to be
\begin{itemize}
\item[(a)] \textit{monotone} on $X$ if for any $x,y\in X$,
$$\langle F(x)-F(y),x-y\rangle \ge 0;$$

\item[(b)] \textit{inverse strongly monotone} on $X$ if there exists $\mu >0$ such that for any $x,y\in X$,
$$\langle F(x)-F(y),x-y\rangle \ge \mu ||Fx-Fy||^2;$$

\item[(c)] \textit{pseudomonotone} on $X$ if for any $x,y\in X$,
$$\langle F(x),y-x\rangle \ge 0 \quad\Rightarrow \quad \langle F(y),y-x\rangle \ge 0;$$

\item[(d)] \textit{strongly pseudomonotone} on $X$ if there exists $\mu>0$ such that for any $x,y\in X$,
$$\langle F(x),y-x\rangle \ge 0 \quad\Rightarrow \quad \langle F(y),y-x\rangle \ge \mu ||y-x||^2;$$

\item[(e)] \textit{pseudomonotone}$^{+}$ on $X$ if $F$ is pseudomonotone on $X$ and for any $x,y\in X$,
$$\langle F(x),y-x\rangle \ge 0 \text{ and }\langle F(y),y-x\rangle = 0 \quad \Rightarrow \quad F(x) = F(y).$$
\end{itemize}
\end{Definition}

\begin{Remark}
It is evident that (b) $\Rightarrow$ (a), (a) $\Rightarrow$ (c), (d) $\Rightarrow$ (c) and
(e) $\Rightarrow$ (c). It is also easy to see that property (d) implies that the VIP has at most one solution.
Moreover, if $F$ is pseudomonotone then the solution set of the VIP is closed and convex
(see, e.g., \cite{ALM14,FP03}).
\end{Remark}


\section{Weak Sharp Solutions}


Recall the definition of weak sharp solutions for a variational inequality problem in terms of
Marcotte and Zhu \cite{MZ98}.

The solution set $X^*$ of the VIP is \textit{weakly sharp} provided that $F$ satisfies
\begin{equation}\label{WS}
-F(x^{*}) \in \mathrm{int} \left( \bigcap_{x \in X^{*}} \left[ T_X(x) \cap N_{X^*}(x) \right]^{\circ} \right),
\quad \text{for all } x^{*} \in X^{*}.
\end{equation}
Note that if $X^*$ is weakly sharp, then there exists a constant $\alpha >0$ such that
\begin{equation}\label{T2}
\alpha \mathbb{B} \subset F(x^*) + \left[ T_X(x^*) \cap N_{X^*}(x^*) \right]^{\circ},
\quad \text{for each } x^* \in X^*.
\end{equation}
It is equivalent to say that for each $x^* \in X^*$,
\begin{equation}\label{T3}
\langle F(x^*),v\rangle \ge \alpha \| v \|, \quad \text{for all } v \in T_X(x^*) \cap N_{X^*}(x^*),
\end{equation}
(see proof of Theorem 4.1 in \cite{MZ98}).
We call the constant $\alpha$ in (\ref{T2}) or equivalently in (\ref{T3})
the \textit{modulus of the weak sharpness} of $X^*$.

For the VIP, an error bound is an estimate for the distance from any point in $H$ to the solution set $X^*$.
Marcotte and Zhu \cite{MZ98} showed that if $F$ is continuous and pseudomonotone$^+$ on a compact set $X$,
then $X^*$ is weakly sharp if and only if there exists some $\alpha>0$ such that
$$\alpha\mathrm{dist}(x,X^*) \le G(x), \quad \text{for all } x\in X,$$
where $G$ is the dual gap function associated to VIP and  defined by
\begin{eqnarray}\label{Eq:DGF}
G(x) & :=& \max_{z \in X} \langle F(z), x- z \rangle \nonumber\\
     &  =& \langle F(\tilde{y}), x-\tilde{y} \rangle,
\end{eqnarray}
where $\tilde{y}$ is any point in the set $\Lambda (x) := \mbox{arg}\max_{z \in X} \langle F(z), x- z \rangle$.
We note that the pseudomonotonicity$^+$ of $F$ on $X$ implies that $F$ is constant on $X^*$
(see, e.g., Proposition 2 in \cite{LW15}).

Recently, Liu and Wu \cite{LW15} gave an error bound in term of the primal gap function
for the VIP defined by
\begin{equation}\label{Eq:PGF}
\begin{array}{ll}
g(x) & := \max_{z \in X} \langle F(x), x- z \rangle\\
     &  = \langle F(x), x-z \rangle \quad \mbox{for } z \in \Gamma(x),
\end{array}
\end{equation}
where $\Gamma(x) := \{ z \in X : \langle F(x), x-z \rangle = g(x) \}$ for $x \in H$.
They showed that if $F$ is monotone on $X$ and constant on $\Gamma(x^*)$ for some $x^*\in X^*$, $g(x) < +\infty$,
$g$ is G\^ateaux differentiable and locally Lipschitz on $X^*$,
then $X^*$ is weakly sharp if and only if there is some $\alpha >0$ such that
$$\alpha\mathrm{dist}(x,X^*) \le g(x), \quad \text{for all } x\in X.$$
We also note that if $F$ is pseudomonotone on $X$ and constant on $\Gamma(x^*)$ 
for some $x^*\in X^*$ then $F$ is constant on $X^*$ (see Proposiotion 4 in \cite{LW15}).

We give some characterizations of weak sharpness of the solution set of VIP without using 
the dual gap or the primal gap function. 
We note that in \cite{MZ98} (respectively, in \cite{LW15}), the authors gave an error bound 
in term of the dual gap function $G$ (respectively, the primal gap function $g$). 
They therefore needed some more assumptions. 
In fact, Marcotte and Zhu \cite{MZ98} assumed that the set $X$ is compact; 
this implies the continuous differentiability of $G$.  
Liu and Wu \cite{LW15} required that $g$ is G\^ateaux differentiable. 
Our proofs follow the lines in \cite{MZ98} and \cite{LW15} but with some modifications.

\begin{Theorem} \label{TH1}
Let $X$ be a nonempty closed convex subset of $H$ and $F : X \to H$ be continuous on $X$ and pseudomonotone$^+$
on $X$. Let the solution set $X^*$ of the VIP be nonempty. Then $X^*$ is weakly sharp if and only if
there exists a positive constant $\alpha$ such that
\begin{equation}\label{T1}
\langle F(P_{X^*}(x)), x - P_{X^*}(x)\rangle \ge \alpha \, \mathrm{dist}(x,X^*),\quad \text{for all } x \in X.
\end{equation}
\end{Theorem}

\begin{proof}
Assume that $X^*$ is weakly sharp and let $x\in X$. Then, we have
$$x-P_{X^*}(x) \in T_X(P_{X^*}(x)) \cap N_{X^*}(P_{X^*}(x))$$
and
$$\| x-P_{X^*}(x) \| = \mathrm{dist}(x,X^*).$$
Thus, by (\ref{T3}), we have
$$\langle F(P_{X^*}(x)), x - P_{X^*}(x)\rangle \ge \alpha\, \| x-P_{X^*}(x) \| =\alpha \, \mathrm{dist}(x,X^*).$$

Conversely, assume that (\ref{T1}) is satisfied for some $\alpha >0$.
We show that (\ref{T2}) holds.
Let $x^*\in X^*$. It is evident that (\ref{T2}) holds if $T_X(x^*)\cap N_{X^*}(x^*) =\{ {\bf 0} \}$.
We now assume that $T_X(x^*)\cap N_{X^*}(x^*) \ne \{ {\bf 0} \}$.
Let ${\bf 0} \ne v\in T_X(x^*)\cap N_{X^*}(x^*)$. Then for each $y^*\in X^*$, we have
$$\langle v,v\rangle >0 \quad \text{and} \quad \langle v,y^*-x^*\rangle \le 0.$$
This implies that the set $X^*$ is separated from $x^*+v$ by the hyperplane
$$H_v = \{x\in H : \langle v,x-x^*\rangle = 0\}.$$
Since $v\in T_X(x^*)$,  for each positive sequence $\{t_k\}$ converging to $0$,
there exists a sequence $\{v_k\}$ converging to $v$ such that $x^* + t_kv_k\in X$ for sufficiently large $k$.
Since $\langle v,v_k\rangle >0$ for sufficiently large $k$, $x^*+t_kv_k$ lies in the open set
$\{x\in H: \langle v,x-x^*\rangle >0 \}$.
Therefore, for sufficiently large $k$, we have
$$\mathrm{dist} \left( x^*+t_kv_k,X^* \right) \ge \mathrm{dist} \left( x^*+t_kv_k,H_v \right)
= \frac{t_k\langle v,v_k\rangle}{\|v\|}.$$
Then by (\ref{T1}), for sufficiently large $k$, we have
\begin{eqnarray*}
\left\langle F(P_{X^*}(x^*+t_kv_k)), x^*+t_kv_k - P_{X^*}(x^*+t_kv_k) \right\rangle
& \ge & \alpha \mathrm{dist} \left( x^*+t_kv_k,X^* \right) \\
& \ge & \alpha t_k\, \frac{\langle v,v_k\rangle}{\|v\|},
\end{eqnarray*}
or, equivalently,
\begin{equation}\label{T4}
\left\langle F(P_{X^*}(x^*+t_kv_k)), v_k + \frac{x^*- P_{X^*}(x^*+t_kv_k)}{t_k} \right\rangle \ge\alpha\, \frac{\langle v,v_k\rangle}{\|v\|}.
\end{equation}
Since $t_k>0$ and $v\in N_{X^*}(x^*)$, one has $x^* = P_{X^*}(x^*+t_kv)$.
Then, by the nonexpansiveness of the projection mapping, we obtain
\begin{eqnarray*}
\left| \left| v_k + \frac{x^*- P_{X^*}(x^*+t_kv_k)}{t_k} - v\right| \right|
&=& \left| \left| v_k -v+ \frac{P_{X^*}(x^*+t_kv)- P_{X^*}(x^*+t_kv_k)}{t_k} \right| \right| \\
&\le& \| v_k-v \| + \| v-v_k \| \\
&=& 2 \| v_k - v \| \to 0 \quad\text{as } k \to \infty.
\end{eqnarray*}
Thus,
$$v_k + \frac{x^*- P_{X^*}(x^*+t_kv_k)}{t_k} \to v \quad\text{as } k \to \infty.$$
Taking the limit as $k\to\infty$ in both sides of (\ref{T4}) and using the continuity of $F$ and $P_{X^*}$,
we obtain
$$\langle F(x^*),v\rangle \ge \alpha \| v \|.$$
It follows that
$$\alpha \mathbb{B} \subseteq F(x^*) + [T_X(x^*) \cap N_{X^*}(x^*)]^o.$$
Since $F$ is pseudomonotone$^+$ on $X$, so it is constant on $X^*$.
This together with the latter inclusion imply that $X^*$ is weakly sharp.

\end{proof}

\begin{Theorem}\label{TH2}
Let $X$ be a nonempty, closed and convex subset of a Hilbert space $H$
and $F : X \to H$ be a mapping.
Assume that the solution set $X^*$ of the VIP is nonempty, closed and convex.
\begin{itemize}
\item[{\rm (a)}] If $X^*$ is weakly sharp and $F$ is monotone,
then there exists a positive constant $\alpha >0$ such that
\begin{equation}\label{T5}
\langle F(x), x - P_{X^*}(x)\rangle \ge \alpha \, \mathrm{dist}(x,X^*), \quad \text{for all } x \in X.
\end{equation}
\item[{\rm (b)}] If $F$ is constant on $X^*$ and continuous on $X$ and (\ref{T5}) holds for some $\alpha >0$,
then $X^*$ is weakly sharp.
\end{itemize}
\end{Theorem}

\begin{proof}
(a)\ Since $X^*$ is weakly sharp, there is a constant $\alpha >0$ such that for all $x^*\in X^*$,
$$\left\langle F(x^*), z \right\rangle \ge \alpha \|z\|, \quad \text{for all } z\in T_X(x^*) \cap N_{X^*}(x^*).$$
For $x\in X$, we have
$$x-P_{X^*}(x) \in T_X(P_{X^*}(x)) \cap N_{X^*}(P_{X^*}(x)),$$
and
$$\|x-P_{X^*}(x)\| = \mathrm{dist}(x,X^*).$$
Thus,
\begin{equation}\label{T6}
\left\langle F(P_{X^*}(x)), x - P_{X^*}(x) \right\rangle
\ge \alpha \, \|x - P_{X^*}(x)\| = \alpha \, \mathrm{dist}(x,X^*).
\end{equation}
Since $F$ is monotone, we have
$$ \left\langle F(P_{X^*}(x)), x - P_{X^*}(x) \right\rangle
\le - \left\langle F(x), P_{X^*}(x) - x \right\rangle.$$
Combining this with the inequality (\ref{T6}), we get
$$\alpha \, \mathrm{dist}(x,X^*) \le \left\langle F(x), x - P_{X^*}(x) \right\rangle.$$

(b)\ Let $x^*\in X^*$. We first show that
\begin{equation}\label{T7}
\alpha \mathbb{B} \subset F(x^*) + \left[ T_X(x^*) \cap N_{X^*}(x^*) \right]^{\circ}.
\end{equation}
This is obvious if $T_X(x^*) \cap N_{X^*}(x^*) = \{ {\bf 0} \}$.
So, we assume that $T_X(x^*) \cap N_{X^*}(x^*) \ne \{ {\bf 0} \}$.
Let ${\bf 0} \ne v\in T_X(x^*) \cap N_{X^*}(x^*)$.
Then for each positive sequence $\{t_k\}$ converging to $0$,
there is a sequence $\{v_k\}$ converging to $v$ and $x^* + t_kv_k \in X$.
As in proof of Theorem \ref{TH1}, we have, for sufficiently large $k$, that
$$\mathrm{dist} \left( x^*+t_kv_k,X^* \right) \ge  \frac{t_k\langle v,v_k\rangle}{\|v\|}.$$
By (\ref{T5}), we get
\begin{eqnarray*}
\left\langle F(x^*+t_kv_k), x^*+t_kv_k - P_{X^*}(x^*+t_kv_k) \right\rangle
& \ge & \alpha \, \mathrm{dist} \left( x^*+t_kv_k,X^* \right)\\
& \ge & \alpha \, t_k \frac{\langle v,v_k\rangle}{\|v\|},
\end{eqnarray*}
or, equivalently,
$$\left\langle  F(x^*+t_kv_k), v_k + \frac{x^*-P_{X^*}(x^*+t_kv_k) }{t_k}\right\rangle
\ge \alpha\frac{\langle v,v_k\rangle}{\|v\|}.$$
As in proof of Theorem \ref{TH1}, letting $k\to \infty$ in the latter inequality, we get
$$\langle F(x^*), v\rangle \ge \alpha \|v\|.$$
Thus, (\ref{T7}) holds. Since $F$ is constant on $X^*$, we conclude that $X^*$ is weakly sharp.

\end{proof}


\section{Finite Termination Property}


In this section, we study the finite termination property of a sequence generated
by an algorithm for finding the solutions of the VIP. In particular, we first establish
an abstract result on the finite termination of the sequences. We then apply such result
to discuss the finite termination property of proximal point method, exact proximal point method
and gradient projection method.

Throughout this section, we assume that $X$ is a nonempty, closed and convex subset of a Hilbert space $H$
and $F : X \to H$ is a mapping. We mention the following result due to Matsushita and Xu \cite{MX13}
which will be used in the sequel.

\begin{Lemma}{\rm \cite{MX13}}
Let $x \in X$, then
$$\max \left\{ \langle v,-F(x)\rangle : v\in T_X(x), \| v \| \le 1 \right\}
= \left\| P_{T_{X}(x)}(-F(x)) \right\|.$$
\end{Lemma}

Our first result of this section is stated as follows.

\begin{Theorem} \label{TH3}
Let $F$ be monotone on $X$, $X^*$ be weakly sharp and $\{x_k\}$ be a sequence in $X$.
Then, $x_k \in X^*$ for all $k$ sufficiently large if and only if
\begin{equation}\label{T12}
\lim_{k\to\infty} P_{T_X (x_k)} \left( -F(x_k) \right) = 0.
\end{equation}
\end{Theorem}

\begin{proof}
If there exists $k_0$ such that $x_k \in X^*$ for all $k\ge k_0$,
then $-F(x_k) \in N_X(x_k)$ for all $k\ge k_0$. Hence, (\ref{T12}) holds trivially.

Suppose, on the contrary, that the inverse implication does not hold.
Then there exists a subsequence of $\{x_k\}$ which is still denoted by $\{x_k\}$ such that
$x_k \not\in X^*$ for all $k$.
For each $k$, set $y_k := P_{X^*}(x_k) \in X^*$.
Then, we have $x_k - y_k \in T_X(y_k) \cap N_{X^*}(y_k)$.
By Theorem \ref{TH1}, there exists $\alpha > 0$ such that
$$\alpha \| x_k - y_k \| = \alpha \, \mathrm{dist}(x_k,X^*) \le \langle F(y_k), x_k - y_k\rangle,
\quad \mbox{for all } k.$$
Thus, by the monotonicity of $F$, we have
\begin{eqnarray*}
\alpha &\le&  \left\langle F(y_k), \frac{x_k-y_k}{||x_k-y_k||}\right\rangle\\
&=& \left\langle -F(x_k), \frac{x_k-y_k}{||x_k-y_k||}\right\rangle
+\frac{1}{||x_k-y_k||} \langle F(x_k) - F(y_k),y_k - x_k\rangle\\
&\le&  \max \left\{ \langle v,-F(x_k)\rangle : v\in T_X(x_k), \, \|v\| \le 1 \right\} \\
&=& P_{T_X (x_k)} \left( -F(x_k) \right).
\end{eqnarray*}
Letting $k\to \infty$ and using (\ref{T12}), we obtain $\alpha \le 0$ which
contradicts the fact that $\alpha > 0$.
Therefore, $x_k\in X^*$ for all sufficiently large $k$.

\end{proof}

\begin{Remark}
Marcotte and Zhu in \cite{MZ98} obtained the finite termination of an algorithm for the VIP
under assumption that $F$ is pseudomonotone$^+$ and continuous on a compact convex set in $\R^n$.
Xiu ang Zhang \cite{XZ05} improved the result of Marcotte and Zhu \cite{MZ98}
by assuming $F$ is continuous and pseudomonotone.
Zhou and Wang \cite{ZW08} established the finite termination without using any monotonicity property
on the underlying mapping.
All the mentioned results require (\ref{T12}) and the strong convergence of the sequence $\{x_k\}$.
Matsushita and Xu \cite{MX13} relaxed the strong convergence of $\{x_k\}$ by the strong convergence of
$\{P_{X^*}(x_k)\}$ to some point in $X^*$ and, in addition, they assumed that the mapping $F$ is monotone.
Since we do not need the strong convergence of $\{P_{X^*}(x_k)\}$, our result improves the result given by
Matsushita and Xu \cite{MX13}.
\end{Remark}

\subsection{Proximal Point Method}

We now apply our results to study the finite termination property of proximal point method for
solving a monotone variational inequality. We consider the following proximal point algorithm \cite{R76}
for solving the VIP: $x_1 \in H$ and
\begin{equation}\label{T8}
x_{n+1} = J_{\gamma_n}(x_n+e_n), \quad n = 1,2,\dots ,
\end{equation}
where $\gamma_n \in (0,\infty)$, $e_n \in H$ and  $J_{\gamma_n} = (I+\gamma_n T)^{-1}$ with
$T: H \rightrightarrows H$ defined by
\begin{equation}\label{T9}
T(x) : = \left\{
  \begin{array}{lll}
   F(x) + N_X(x), && \quad \text{if } x\in X,\\
    \emptyset,    && \quad \text{otherwise}.
  \end{array}
  \right.
\end{equation}
Note that, for all $x\in H$ and $r>0$, the inclusion
\begin{equation}\label{T10}
x\in x_r + rT(x_r)
\end{equation}
has a unique solution $x_r\in H$ (see, e.g., \cite{R70}).

From (\ref{T9}) and (\ref{T10}), we have
$$x_n + e_n \in \left( I + \gamma_n(F+N_X) \right) (x_{n+1}), \quad \mbox{for all }n. $$
Then
$$x_n + e_n - x_{n+1} - \gamma_n F(x_{n+1}) \in \gamma_n N_X(x_{n+1}) = N_X(x_{n+1}).$$
It follows that
\begin{equation} \label{PA}
\langle x_n + e_n  - x_{n+1} - \gamma_n F(x_{n+1}), y - x_{n+1}\rangle \le 0\qquad \text{for all} \,\,y\in X.
\end{equation}
This is equivalent to
\begin{equation}\label{PPA}
x_{n+1} = P_X(x_n - \gamma_nF(x_{n+1}) + e_n),\quad n=1,2,\dots
\end{equation}

The following result provides the finite convergence property of proximal point method
(\ref{PPA}).

\begin{Theorem} \label{TH4}
Let $F$ be monotone and $\{x_n\}$ be a sequence generated by (\ref{PPA}) such that
$\liminf_{n\to\infty}\gamma_n >0$. Suppose that the following conditions hold:
\begin{itemize}
\item[{\rm (i)}] $\{x_{n+1} - x_n\}$ converges strongly to $0$;
\item[{\rm (ii)}] $\{e_n\}$ converges strongly to $0$.
\end{itemize}
If $X^*$ is weakly sharp, then $x_n \in X^*$ for all sufficiently large $n$.
\end{Theorem}

\begin{proof}
Since
$$\left\langle x_n + e_n  - x_{n+1} - \gamma_n F(x_{n+1}), y - x_{n+1} \right\rangle \le 0,
\quad \text{for all } y\in X,$$
we have, for all $y \in X$, that
\begin{eqnarray}\label{T101}
\left\langle F(x_{n+1}), x_{n+1} - y \right\rangle
&\le& \frac{1}{\gamma_n} \left( \langle x_n - x_{n+1},x_{n+1}-y\rangle
+ \langle e_{n}, x_{n+1}-y \rangle \right)\nonumber\\
&\le& \frac{1}{\gamma_n} \left( \|x_n - x_{n+1}\| \, \|x_{n+1} - y\| + \|e_{n}\| \, \|x_{n+1}-y\| \right).
\end{eqnarray}
Assume that the conclusion is false. Then, there exists a subsequence $\{x_{n_i}\}$ of $\{x_n\}$
such that $x_{n_i}\not\in X^*$ for all $i$. For each $i$, set $y_{n_i} = P_{X^*}(x_{n_i})$.
Then  $x_{n_i}\ne y_{n_i}$ for all $i$. By Theorem \ref{TH2} (a), for some $\alpha >0$, we have
\begin{equation}\label{T111}
\left\langle F(x_{n_i}), x_{n_i}-y_{n_i} \right\rangle
\ge \alpha \, \mathrm{dist}(x_{n_i},X^*) = \alpha \, \|x_{n_i} - y_{n_i}\|.
\end{equation}
Taking $n: = n_i$ and $y: = y_{n_i+1}$ in (\ref{T101}) and using (\ref{T111}), we get
$$\alpha \|x_{n_i+1} - y_{n_i+1}\| \le \frac{1}{\gamma_{n_i}}
\left( \|x_{n_i}-x_{n_i+1}\| \, \|x_{n_i+1}-y_{n_i+1}\|
+ \|e_{n_i}\| \, \|x_{n_i+1}-y_{n_i+1}\| \right),$$
and then
$$\alpha \le \frac{1}{\gamma_{n_i}}  \left( \|x_{n_i}-x_{n_i+1}\| + \|e_{n_i}\| \right).$$
Since $\liminf_{n \to \infty} \gamma_n >0$, $\lim_{n\to\infty} e_n = 0$ and
$x_{n+1}-x_n \to 0$ as $n\to \infty$,
it follows from the latter inequality that $\alpha \le 0$ which contradicts the fact that $\alpha > 0$.
Therefore, $x_n\in X^*$ for all sufficiently large $n$.

\end{proof}

\begin{Remark}
Theorem \ref{TH4} is an improvement of Theorem 3.1 in \cite{MX15} because we do not require
that $\{P_{X^*}(x_k)\}$ converges strongly to some point in $X^*$.
\end{Remark}

\subsection{Exact Proximal Point Method}

We now study the finite termination propery for the  exact proximal point method:
\begin{equation}\label{T112}
x_{n+1} = P_X \left( x_n - \gamma_n F(x_{n+1}) \right), \quad n = 1,2,\ldots
\end{equation}
This is the case in (\ref{PPA}) we take  $e_n = 0,\, n = 1,2,\ldots$.
\begin{Theorem}\label{TH5}
Let $F : X \to H$ be monotone and $X^*$ be weakly sharp with modulus $\alpha>0$.
Let $\{x_n\}$ be the sequence generated by (\ref{T112}) with, for some positive number $a$, $\gamma_n \in [a,+\infty)$ for all $n$.
Then, $\{x_n\}$ converges to a point in $X^*$ in atmost $\ell$ iterations with
$$\ell \le \frac{\mathrm{dist(x_1,X^*)^2}}{a^2\alpha^2} + 1.$$
\end{Theorem}

\begin{proof}
From Rockafellar \cite{R76}, we know that for $x^*\in X^*$
\begin{equation}\label{T113}
\|x_{n+1}-x^*\|^2 \le \|x_n - x^*\|^2 - \|x_{n+1}-x_{n}\|^2, \quad n=1,2, \ldots .
\end{equation}
This implies that $\lim_{n\to \infty} \|x_n-x^*\|$ exists
and $\lim_{n\to\infty}\|x_{n+1}-x_n\| = 0$. By Theorem \ref{TH4}, $x_n\in X^*$ for all sufficiently large $n$.

Now, for $1<N \in \mathbb{N}$, we have from (\ref{T113}) that
\begin{eqnarray*}
\|x_1-x^*\|^2
&\ge& \|x_2-x^*\|^2 + \|x_2-x_1\|^2\\
&\ge& \|x_3-x^*\|^2 + \|x_3-x_2\|^2 + \|x_2-x_1\|^2\\
&\vdots& \\
&\ge& \|x_{N+1}-x^*\|^2 + \sum_{i=1}^{N} \|x_{i+1}-x_i\|^2\\
&\ge&  \sum_{i=1}^{N} \|x_{i+1}-x_i\|^2.
\end{eqnarray*}
Therefore, for all $N>1$,
\begin{equation}\label{T114}
\mathrm{dist} \left( x_1,X^* \right)^2 = \inf_{x^*\in X^*} ||x_1 - x^*||^2 \ge \sum_{i=1}^{N} ||x_{i+1}-x_i||^2.
\end{equation}
Since $\lim_{n\to \infty} \|x_{n+1}-x_{n}\| = 0$,
we let $\ell$ be the smallest integer such that $||x_{\ell+1}-x_\ell|| < a \, \alpha$.

If $x_{\ell+1} \not\in X^*$, then we set $y_{\ell+1} = P_{X^*}(x_{\ell+1})$.
By Theorem \ref{TH2} (a) and (\ref{PA}), we have
\begin{eqnarray*}
\alpha \|x_{\ell+1} - y_{\ell+1}\|
&=& \alpha \, \mathrm{dist} \left( x_{\ell+1},X^* \right) \\
&\le& \left\langle F(x_{\ell+1}),x_{\ell+1}-y_{\ell+1} \right\rangle\\
&\le& \frac{1}{\gamma_{\ell}} \langle x_\ell - x_{\ell+1},x_{\ell+1} - y_{\ell+1}\rangle\\
&\le& \frac{1}{\gamma_{\ell}} \|x_\ell - x_{\ell+1}\| \, \|x_{\ell+1} - y_{\ell+1}\|.
\end{eqnarray*}
This implies that $\gamma_{\ell} \, \alpha \le \|x_{\ell+1}-x_{\ell}\| < a \, \alpha$,
or, equivalently, $\gamma_{\ell} < a$ which contradicts our assumption that
$\gamma_{\ell} \geq a$.
Thus, $x_{\ell+1} \in X^*$. Hence, we have
$$\mathrm{dist} \left( x_1,X^* \right)^2 \ge
\sum_{i=1}^{\ell-1} \|x_{i+1}-x_i\|^2 \ge (\ell-1)a^2 \, \alpha^2.$$
Therefore,
$$\ell \le \frac{\mathrm{dist} \left( x_1,X^* \right)^2}{a^2\alpha^2} +1.$$

\end{proof}

\begin{Remark}
If follows from Theorem \ref{TH5} that if $\gamma_1$ is chosen to be sufficiently large, i.e.,
$a$ sufficiently large, then the exact proximal point algorithm has one step termination.
Hence, Theorem \ref{TH5} is an extension of Theorem 4.3 in \cite{XZ05}.
\end{Remark}

\subsection{Gradient Projection Method}

We now establish the finite termination property of the following gradient projection method:
\begin{equation}\label{GPM}
x_{n+1} = P_X(x_n - \gamma_nF(x_n)),\qquad n = 1,2,\ldots .
\end{equation}
This is the case when we take $e_n:= \gamma_n(F(x_{n+1})-F(x_n))$, $n=1,2,\ldots$ in (\ref{PPA}).

The finite termination property for the method (\ref{GPM}) was studied in \cite{MX15,XZ02}.
Xiu and Zhang \cite{XZ02}  established the finite termination of a modification of (\ref{GPM})
by adding a vanishing error term under the nondegeneracy assumption. To obtain their results,
they also assumed the strongly convergence of the sequence generated by the method.
Matsushita and Xu \cite{MX15} dropped the assumption that the sequence converges strongly and
obtained the finite termination property of (\ref{GPM}) in Hilbert spaces under the assumption that
$F$ is inverse strongly monotone. Here, we study the finite termination property of (\ref{GPM})
under the assumption that $F$ is strongly pseudomonotone and monotone at the same time.

Assume $F$ is strongly pseudomonotone. Then the VIP has unique solution, say $x^*$.
If, in addition, $F$ is monotone and $X^* = \{x^*\}$ is weakly sharp with modulus $\alpha$, then by Theorem \ref{TH2} (a), we have
$$\alpha \|x-x^*\| \le \langle F(x), x-x^*\rangle,\quad \mbox{for all } x\in X.$$

\begin{Lemma}\label{LM1}
Let $F : X \to H$ be strongly pseudomonotone with the modulus $\mu$ and Lipschitz continuous with the constant $L$.
Let $\{x_n\}$ be the sequence generated by (\ref{GPM}).
If $x^*$ is a unique solution of the VIP, then
\begin{equation}\label{T115}
\left[ 1+\gamma_n(2\mu - \gamma_nL^2) \right] \, \|x_{n+1}-x^*\|^2 \le \|x_n-x^*\|^2,
\,\, \mbox{for all} \,\,n= 1,2,\ldots
\end{equation}
and
\begin{equation}\label{T116}
\left( 1-\gamma_n^2 \right) \|x_{n+1}-x_n\|^2 \le
\left( L^2-2\gamma_n\mu-1 \right) \, \|x_{n+1}-x^*\|^2 + \|x_n-x^*\|^2,
\,\, \mbox{for all }\,\, n= 1,2,\ldots
\end{equation}
\end{Lemma}

\begin{proof}
We only prove (\ref{T116}). For the proof of (\ref{T115}), see \cite{KV14}.

Since $e_n = \gamma_n(F(x_{n+1})-F(x_n))$, it follows from (\ref{PA}) that
$$ \left\langle x_n - \gamma_nF(x_n) - x_{n+1}, y-x_{n+1} \right\rangle \le 0,\quad \mbox{for all } y\in X.$$
Taking $y=x^*$ in the latter inequality, one obtains
\begin{equation}\label{T117}
\left\langle x_n - x_{n+1},x^* - x_{n+1} \right\rangle \le
\gamma_n \left\langle F(x_n),x^*-x_{n+1} \right\rangle.
\end{equation}
Since $x^*\in X^*$, we have
$$\left\langle F(x^*),x-x^* \right\rangle \ge 0, \quad \mbox{for all } x\in X.$$
Hence, by the strong pseudomonotonicity of $F$, we get
$$\langle F(x),x-x^*\rangle \ge \mu||x-x^*||^2, \quad \mbox{for all } x\in X.$$
By the Cauchy-Schwarz inequality and the Lipschitz continuity of $F$, we obtain
\begin{eqnarray}\label{T118}
\gamma_n\langle F(x_n),x^*-x_{n+1}\rangle
&=& -\gamma_n \langle F(x_{n+1}),x^*-x_{n+1}\rangle +\gamma_n\langle F(x_n)-F(x_{n+1}),x^*-x_{n+1}\rangle \nonumber\\
&\le& -\gamma_n\mu \|x_{n+1}-x^*\|^2 +\gamma_n \|F(x_n)-F(x_{n+1})\| \, \|x^*-x_{n+1}\| \nonumber\\
&\le& -\gamma_n\mu \|x_{n+1}-x^*\|^2 +\gamma_n L \|x_n-x_{n+1}\| \, \|x^*-x_{n+1}\| \nonumber\\
&\le& -\gamma_n\mu \|x_{n+1}-x^*\|^2 +\frac{1}{2} \gamma_n^2 \|x_n-x_{n+1}\|^2
+ \frac{1}{2} L^2 \|x^*-x_{n+1}\|^2 \nonumber\\
&=& \frac{1}{2} \left[(L^2-2\gamma_n\mu)\|x_{n+1}-x^*\|^2 + \gamma_n^2 \|x_n-x_{n+1}\|^2 \right].
\end{eqnarray}
On the other hand,
\begin{equation}\label{T119}
\langle x_n - x_{n+1},x^* - x_{n+1}\rangle =
\frac{1}{2} \left( \|x_{n+1}-x_n\|^2 + \|x_{n+1}-x^*\|^2 - \|x_n-x^*\|^2 \right).
\end{equation}
From (\ref{T117}) - (\ref{T119}), we have
$$\|x_{n+1}-x_n\|^2 + \|x_{n+1}-x^*\|^2 - \|x_n-x^*\|^2
\le (L^2-2\gamma_n\mu)\|x_{n+1}-x^*\|^2 +\gamma_n^2||x_n-x_{n+1}||^2.$$
Then, we obtain (\ref{T116}).

\end{proof}

\begin{Theorem}\label{TH6}
Let $F : X \to H$ be strongly pseudomonotone with the modulus $\mu$ and Lipschitz with the constant $L$.
Let $\{x_n\}$ be the sequence generated by (\ref{GPM}). Suppose that
\begin{equation}\label{T120}
\frac{L^2}{2\mu} \le \gamma_n\le \sigma <1, \quad \mbox{for all } n = 1,2,\ldots
\end{equation}
where $\sigma$ is a positive constant.
Assume further that $F$ is monotone and $X^*= \{x^*\}$ is weakly sharp with modulus $\alpha$.
Then, $\{x_n\}$ converges to $x^*$ in atmost $\ell$ iterates with
$$\ell \le \frac{\|x_1-x^*\|^2(2\mu+L^3)^2}{(1-\sigma^2)\alpha^2L^4}+1.$$
\end{Theorem}

\begin{proof}
It follows from (\ref{T120}) that $2\mu - \gamma_n L^2 \ge 0$ for all $n$.
Hence, (\ref{T115}) implies
$$\|x_{n+1} - x^*\| \le \|x_n-x^*\|, \quad \mbox{for all } n,$$
and then $\lim_{n\to\infty}\|x_n - x^*\| = a\in [0,\infty)$.
By (\ref{T120}) and (\ref{T116}), we have
\begin{equation}\label{T121}
(1-\sigma^2)\|x_{n+1}-x_n\|^2 \le \|x_n-x^*\|^2 - \|x_{n+1}-x^*\|^2, \quad \mbox{for all } n.
\end{equation}
Using (\ref{T121}) and progressing as in the proof of Theorem \ref{TH5}, we obtain
\begin{equation}\label{T122}
\|x_1-x^*\|^2 \ge \left( 1-\sigma^2 \right) \sum_{i=1}^{N}\|x_{i+1}-x_i\|^2,\quad \text{for all } N>1.
\end{equation}
Moreover, letting $n\to \infty$ in the both sides of (\ref{T121}), we get
$$\lim_{n\to\infty}\|x_{n+1}-x_n\| = 0.$$
Let $\ell$ be the smallest integer such that
\begin{equation}\label{T123}
\| x_{\ell+1}-x_\ell \| < \frac{L^2\alpha}{2\mu+L^3}.
\end{equation}
If $x_{\ell+1}\ne x^*$, then by the weak sharpness of the solution set $X^{*}$,
the Lipschitz continuity of $F$ and (\ref{T101}), we have
\begin{eqnarray*}
\alpha\|x_{\ell+1}-x^*\|
&\le& \langle F(x_{\ell+1}), x_{\ell+1}-x^*\rangle\\
&\le& \frac{1}{\gamma_{\ell}}(\|x_\ell - x_{\ell+1}\| \, \|x_{\ell+1}-x^*\| + \|e_\ell \| \, \|x_{\ell+1}-x^*\|)\\
&=&\frac{1}{\gamma_{\ell}}(\|x_\ell - x_{\ell+1}\| \, \|x_{\ell+1}-x^*\|
+ \gamma_{\ell} \|F(x_\ell) - F(x_{\ell+1}) \| \, \|x_{\ell+1}-x^*\|)\\
&\le& \left(\frac{1}{\gamma_\ell} +  L\right) \|x_\ell - x_{\ell+1}\| \, \|x_{\ell+1}-x^*\|\\
&\le& \left(\frac{2\mu}{L^2} + L\right) \|x_\ell - x_{\ell+1}\| \, \|x_{\ell+1}-x^*\|
\end{eqnarray*}
This implies
$$\|x_{\ell+1}-x_\ell \| \ge \frac{\alpha L^2}{2\mu + L^3},$$
which contradicts (\ref{T123}).
Thus, $x_{\ell+1} = x^*$.
It follows from (\ref{T121}) that
$$\|x_1-x^*\|^2 \ge (1-\sigma^2)\sum_{i=1}^{\ell-1} \|x_{i+1}-x_i\|^2
\ge (1-\sigma^2)(\ell-1) \frac{\alpha^2L^4}{(2\mu + L^3)^2},$$
and so,
$$\ell \le \frac{\|x_1-x^*\|^2(2\mu+L^3)^2}{(1-\sigma^2)\alpha^2L^4}+1.$$
\end{proof}


\textbf{Acknowledgements.}
In this research, first author was funded by the National Plan for Science, Technology and
Innovation (MAARIFAH) - King Abdulaziz City for Science and Technology - through
the Science \& Technology Unit at King Fahd University of Petroleum \& Minerals (KFUPM) - the
Kingdom of Saudi Arabia, award number 12-MAT3023-24, Luong V. Nguyen was supported by funds allocated to the implementation of the international co-funded project in the years 2014-2018, 3038/7.PR/2014/2, and by the EU grant PCOFUND-GA-2012-600415. The paper was completed while the third author was visiting Mathematics \& Statistics Department,  KFUPM. He thanks the Department for hospitality.


\end{document}